\documentclass[12pt]{article}
\usepackage{a4}
\usepackage{amsthm}
\usepackage{amsmath}
\usepackage{amssymb}
\usepackage{amsfonts}
\usepackage[scale=.73]{geometry}
\usepackage{stmaryrd}
\usepackage{cite}
\usepackage{epsfig}
\usepackage{mathtools}
\usepackage{xcolor}
\newtheorem{conjecture}{Conjecture}
\newtheorem{theorem}{Theorem}
\newtheorem{lemma}[theorem]{Lemma}
\newtheorem{proposition}[theorem]{Proposition}

\newcommand\CC{{\cal C}}
\newcommand\RR{{\cal R}}
\newcommand\SSS{{\cal S}}
\newcommand\TT{{\cal T}}
\newcommand\NN{{\mathbb N}}
\newcommand\ZZ{{\mathbb Z}}
\def\FOlocal{\mathrm{FO}_1^{\mathrm{local}}}

\begin{document}
\title{Strong modeling limits of graphs with bounded tree-width\thanks{This work received funding from the European Research Council (ERC) under the European Union’s Horizon 2020 research and innovation programme (grant agreement No 648509). The second and third authors were also supported by the MUNI Award in Science and Humanities of the Grant Agency of Masaryk University (MUNI/I/1677/2018). This publication reflects only its authors' view; the European Research Council Executive Agency is not responsible for any use that may be made of the information it contains.\newline An extended abstract containing the results presented in this paper has appeared in the proceeding of EuroComb'21.}}
\author{Andrzej Grzesik\thanks{Faculty of Mathematics and Computer Science, Jagiellonian University, {\L}ojasiewicza 6, 30-348 Krak\'{o}w, Poland. E-mail: {\tt Andrzej.Grzesik@uj.edu.pl}.}\and
        Daniel Kr{\'a}l'\thanks{Institute of Mathematics, Leipzig University, Augustusplatz 10, 04109 Leipzig, and Max Planck Institute for Mathematics in the Sciences, Inselstra{\ss}e 22, 04103 Leipzig, Germany. E-mail: {\tt daniel.kral@uni-leipzig.de}.}{\hskip 0.75ex}\thanks{Previous affiliation: Faculty of Informatics, Masaryk University, Botanick\'a 68A, 602 00 Brno, Czech Republic.}\and
        \newcounter{lth}
	\setcounter{lth}{4}
	Samuel Mohr$^\fnsymbol{lth}$}
	
\date{}
\maketitle
\begin{abstract}
The notion of first order convergence of graphs unifies the notions of convergence for sparse and dense graphs.
Ne\v set\v ril and Ossona de Mendez [J. Symbolic Logic 84 (2019), 452--472] proved that
every first order convergent sequence of graphs from a nowhere-dense class of graphs has a modeling limit and
conjectured the existence of such modeling limits with an additional property,
the strong finitary mass transport principle.
The existence of modeling limits satisfying the strong finitary mass transport principle
was proved for first order convergent sequences of trees by Ne\v set\v ril and Ossona de Mendez [Electron. J. Combin. 23 (2016), P2.52] and
for first order sequences of graphs with bounded path-width by Gajarsk\'y et al.~[Random Structures Algorithms 50 (2017), 612--635].
We establish the existence of modeling limits satisfying the strong finitary mass transport principle
for first order convergent sequences of graphs with bounded tree-width.
\end{abstract}


\section{Introduction}
\label{sec:intro}

The theory of combinatorial limits is an evolving area of combinatorics.
The most developed is the theory of graph limits,
which is covered in detail in a recent monograph by Lov\'asz~\cite{Lov12}.
Further results concerning many other combinatorial structures exist,
e.g.\ for permutations~\cite{GleGKK15,HopKMRS13,HopKMS11,KraP13,ChaKNPSV20,KenKRW20,Kur22} or for partial orders~\cite{HlaMPP15,Jan11}. 
In the case of graphs, limits of dense graphs~\cite{BorCLSSV06,BorCLSV08,BorCLSV12,LovS06,LovS10},
also see~\cite{CorR23,CorR20} for a general theory of limits of dense combinatorial structures, and
limits of sparse graphs~\cite{AldL07,BenS01,BolR11,BorCG17,Ele07,HatLS14} evolved to a large extent independently.
A notion of first order convergence was introduced by Ne\v set\v ril and Ossona de Mendez~\cite{NesO16,NesO20}
as an attempt to unify convergence notions in the dense and sparse regimes.
This general notion can be applied in the setting of any relational structures,
see e.g.~\cite{HosNO17} for results on limits of mappings or~\cite{KarKLM17} on matroids.
Informally speaking,
a sequence of relational structures is \emph{first order convergent}
if for any first order property, the density of $\ell$-tuples of the elements having this property converges;
a formal definition is given in Subsection~\ref{subsec:fo}.
Every first order convergent sequence of dense graphs
is convergent in the sense of dense graph convergence from~\cite{BorCLSV08,BorCLSV12}, and
every first order convergent sequence of graphs with bounded degree
is convergent in the sense of Benjamini-Schramm convergence as defined in~\cite{BenS01}.

A first order convergent sequence of graphs can be associated with an analytic limit object,
which is referred to as a \emph{modeling limit} (see Subsection~\ref{subsec:fo} for a formal definition).
However, not every first order convergent sequence of graphs has a modeling limit~\cite{NesO20} and
establishing the existence of a modeling limit for first order convergent sequences of graphs
is an important problem in relation to first order convergence of graphs:
a modeling limit of a first order convergent sequence of dense graphs yields a graphon,
the standard limit object for convergent sequences of dense graphs, and
a modeling limit of a first order convergent sequence of sparse graphs that satisfies 
the strong finitary mass transport principle (see Subsection~\ref{subsec:fo} for the definition of the principle)
yields a graphing, the standard limit object for convergent sequence of sparse graphs.

Ne\v set\v ril and Ossona de Mendez~\cite{NesO20} conjectured that
every first order convergent sequence of graphs from a nowhere-dense class of graphs has a modeling limit.
Nowhere-dense classes of graphs include many sparse classes of graphs,
in particular, classes of graphs with bounded degree and minor closed classes of graphs;
see~\cite{NesO08a,NesO08b,NesO08c,NesO11,NesO12book} for further details and many applications.
The existence of modeling limits for convergent sequences of graphs from a monotone nowhere-dense class of graphs
was proven in~\cite{NesO19}.
\begin{theorem}[Ne\v set\v ril and Ossona de Mendez~\cite{NesO19}]
\label{thm:general}
Let $\CC$ be a monotone class of graphs.
Every first order convergent sequence of graphs from $\CC$ has a modeling limit if and only if
$\CC$ is nowhere-dense.
\end{theorem}
Theorem~\ref{thm:general} gives little control on the measure of vertex subsets in a modeling limit,
which naturally have the same size in finite graphs, e.g., those joined by a perfect matching.
The strong finitary mass transport principle, vaguely speaking,
translates natural constraints on sizes of vertex subsets to measures of corresponding vertex subsets in a modeling limit.
We refer to Subsection~\ref{subsec:fo} for further details.

Ne\v set\v ril and Ossona de Mendez~\cite{NesO19} conjectured that
Theorem~\ref{thm:general} can be strengthened by adding a condition that
modeling limits satisfy the strong finitary mass transport principle.
\begin{conjecture}[{Ne\v set\v ril and Ossona de Mendez~\cite[Conjecture 6.1]{NesO19}}]
Let $\CC$ be a nowhere-dense monotone class of graphs.
Every first order convergent sequence of graphs from $\CC$ has a modeling limit that
satisfies the strong finitary mass transport principle.
\end{conjecture}
The existence of modeling limits satisfying the strong finitary mass transport principle
is known for first order convergent sequences
of trees of bounded depth and more generally sequences of graphs with bounded tree-depth~\cite{NesO20},
sequences of trees~\cite{NesO16} and sequences of graphs with bounded path-width~\cite{GajHKKKOOT17},
which can be interpreted in plane trees.
Our main result (Theorem~\ref{thm:main}) establishes
the existence of modeling limits satisfying the strong finitary mass transport principle
for sequences of graphs with bounded tree-width.
\newcounter{storemaintheorem}
\setcounter{storemaintheorem}{\thetheorem}
\begin{theorem}
\label{thm:main}
Let $k$ be a positive integer.
Every first-order convergent sequence of graphs with tree-width at most $k$
has a modeling limit satisfying the strong finitary mass transport principle.
\end{theorem}
While it may seem at the first sight that a proof of Theorem~\ref{thm:main} can be an easy combination
of a proof of the existence of modeling limits satisfying the strong finitary mass transport principle
for trees from~\cite{NesO16} and for graphs with bounded path-width~\cite{GajHKKKOOT17},
this is actually not the case.
In fact, the argument in~\cite{GajHKKKOOT17} is based on interpretation of graphs with bounded path-width in plane trees,
i.e., the results in both~\cite{GajHKKKOOT17} and~\cite{NesO16}
on the existence of modeling limits satisfying the strong finitary mass transport principle
do not go significantly beyond the class of trees. 

We have not been able to find a first order interpretation of graphs with bounded tree-width in (plane) trees, and
we believe that this is related to a possibly complex structure of vertex cuts in such graphs,
which need to be addressed using a more general approach.
Specifically, the proof of Theorem~\ref{thm:main} is based on constructing modeling limits of rooted $k$-trees,
whose orientation essentially encodes the universal weak coloring orders studied in relation to sparse classes of graphs~\cite{NesO12book},
so the proof may be amenable to an extension to graph classes with bounded expansion in principle.
We remark that our arguments can be easily adapted to show
the existence of modeling limits of first order convergent sequences of graphs with bounded tree-width that are residual (in the sense of~\cite{NesO16,NesO19}),
which could yield an alternative proof of Theorem~\ref{thm:main}
when combined with the framework described in~\cite[Theorem 1]{NesO16}.

The proof of Theorem~\ref{thm:main}, similarly to the proof of the existence of modeling limits of plane trees in~\cite{GajHKKKOOT17},
has two steps: the decomposition step, focused on distilling first order properties of graphs in the sequence, and
the composition step, focused on constructing a modeling limit consistent with the identified first order properties.
These two steps also appear implicitly in~\cite{NesO16,NesO20},
in particular, the decomposition step is strongly related to the comb structure results presented in~\cite{NesO16,NesO20}.
The arguments of the decomposition step of the proof of Theorem~\ref{thm:main} are analogous to those used in~\cite[Subsection 3.1]{GajHKKKOOT17}.
The composition step
however requires a conceptual extension of techniques used for modeling limits of trees as
we had to deal with vertex separations of sizes larger than one.
This was achieved by a careful analysis of different types of paths arising in the orientation corresponding to a weak coloring order.
This analysis allows defining the edge set of a modeling limit in a measurable and consistent way for vertex separations of sizes larger than one.

The paper is organized as follows.
In Section~\ref{sec:notation}, we introduce notation used in the paper, in particular,
notions related to graphs with bounded tree-width, first order convergence, and model theory.
In Section~\ref{sec:decompose}, we overview the decomposition step from~\cite{GajHKKKOOT17} and
phrase the presented results in the context of graphs with bounded tree-width.
The core of the paper is Section~\ref{sec:compose}
where we construct modeling limits of edge-colored rooted $k$-trees that satisfy the strong finitary mass transport principle.
This construction is then used in Section~\ref{sec:main} to prove Theorem~\ref{thm:main}.

\section{Notation}
\label{sec:notation}

We first introduce notation specific to this paper.
For the notions not defined here, we refer the reader to \cite{Die10} and~\cite{EbbF05}
for the graph theory terminology and the model theory terminology, respectively.
Some of the less common general notation that we use here include the following.
The set of the first $k$ positive integers is denoted by $[k]$ and $\NN_0$ denotes the set of all non-negative integers.
If $x$ is a real number and $z$ is a positive real,
we write $x\mod z$ for the unique real $x'\in [0,z)$ such that $x=x'+kz$ for some $k\in\ZZ$.
Finally, if $G$ is a graph, then $|G|$ stands for the number of vertices of $G$.

A graph has tree-width at most $k$ if and only if it is a subgraph of a $k$-tree.
A \emph{$k$-tree} can be defined recursively as follows:
each complete graph with at most $k$ vertices is a $k$-tree, and
if $G$ is a $k$-tree, then any graph obtained from $G$ by adding a vertex adjacent to $k$ vertices forming a complete subgraph is also $k$-tree.
We next define a notion of a \emph{rooted $k$-tree}; the definition is also recursive.
Any transitive tournament with at most $k$ vertices is a rooted $k$-tree, and
if $G$ is a rooted $k$-tree and $v_1,\ldots,v_k$ are vertices that form a transitive tournament,
then the graph obtained from $G$ by adding a new vertex $v$ and adding an edge directed from $v$ to $v_i$ for every $i\in [k]$
is also a rooted $k$-tree.
Observe that every rooted $k$-tree is an acyclic orientation of a $k$-tree (the converse need not be true).
In the setting above, assume that $v_1\cdots v_k$ is a directed path;
we say that the vertex $v_i$, $i\in [k]$, is the \emph{$i$-parent} of the vertex $v$ and
the vertex $v$ is an \emph{$i$-child} of $v_i$.
We extend this notation to the initial tournament by setting the out-neighbors of each vertex to be its $1$-parent, $2$-parent, etc.\ 
in a way that a vertex with $\ell$ out-neighbors has an $i$-parent for every $i\in [\ell]$ and
there is an edge from its $i$-parent to its $i'$-parent for every $i<i'$.
An edge from $v$ to its $i$-parent is referred to as an \emph{$i$-edge}.
Hence, every edge is an $i$-edge for some $i\in [k]$.
To simplify our exposition, we say that $e$ is an \emph{$A$-edge} for $A\subseteq [k]$ if $e$ is an $i$-edge for some $i\in A$.
Finally, when $i$ is not important for our considerations,
we may just say that $v'$ is a \emph{parent} of $v$, and $v$ is a \emph{child} of $v'$, if there is a directed edge from $v$ to $v'$,
i.e., when $v'$ is an $i$-parent of $v$ for some $i\in [k]$.

We state several simple properties of rooted $k$-trees in the form of propositions
to be able to refer to them later in our exposition.

\begin{proposition}
\label{prop:tw}
The tree-width of a graph $G$ is the minimum $k$ for which there exists
an orientation of $G$ that is a spanning subgraph of a rooted $k$-tree.
\end{proposition}

\begin{proposition}
\label{prop:parent}
Let $G$ be a rooted $k$-tree, $v$ a vertex of $G$, and $w$ and $w'$ the $i$-parent and $i'$-parent of $v$ for some $i<i'$.
Then the vertex $w'$ is an $i''$-parent of $w$ for some $i''\le i'$.
\end{proposition}

In our exposition, we will consider rooted $k$-trees with edges colored with two colors,
which we will refer to as \emph{$2$-edge-colored $k$-trees}.
Hence, each edge of a $2$-edge-colored $k$-tree is an $i$-edge for some $i\in [k]$ and also has one of the two colors.
In the proof of our main result, given a convergent sequences of graphs $(G_n)_{n\in\NN}$ with tree-width $k$,
we construct modeling limits for a convergent sequence of $2$-edge-colored $k$-trees such that
the graphs $G_n$ are spanning subgraphs of the $2$-edge-colored $k$-trees. 
The $2$-coloring of the edges  of the $2$-edge-colored $k$-trees encodes 
which edges of the $k$-trees of the sequence are also edges of the graphs $G_n$. 

\subsection{First order convergence}
\label{subsec:fo}

We now formally define the notion of first order convergence.
This notion can be used for all relational structures and beyond, e.g., matroids~\cite{KarKLM17},
however, for simplicity, we limit our exposition to graphs, which may (but need not) be directed and edge-colored.
In particular, in the case of $2$-edge-colored $k$-trees,
we consider relational structures with $k+2$ binary relations such that
$k$ of them encode the relation between vertices and their $i$-parents, $i\in [k]$, and
two binary relations encode the edge-coloring with two colors.
If $\psi$ is a first order formula with $\ell$ free variables and $G$ is a (finite) graph,
then the \emph{Stone pairing} $\langle \psi,G\rangle$ is the probability that
a uniformly chosen $\ell$-tuple of vertices of $G$ satisfies $\psi$.
A sequence $(G_n)_{n\in\NN}$ of graphs is \emph{first order convergent}
if the limit $\lim\limits_{n\to\infty}\langle \psi,G_n\rangle$ exists for every first order formula $\psi$.
Every sequence of graphs has a first order convergent subsequence, see e.g.~\cite{NesO20,NesO12,NesO16a}.

A \emph{modeling} $M$ is a (finite or infinite) graph 
with a standard Borel space used as its vertex set equipped with a probability measure
such that for every first order formula $\psi$ with $\ell$ free variables,
the set of all $\ell$-tuples of vertices of $M$ satisfying a formula $\psi$
is measurable in the product measure.
In the analogy to the graph case, the \emph{Stone pairing} $\langle \psi,M\rangle$
is the probability that a randomly chosen $\ell$-tuple of vertices satisfies $\psi$,
i.e., $\langle \psi,M\rangle$ is the measure of the set containing $\ell$-tuples of vertices $v_1,\ldots,v_\ell$ such that
$M\models\psi(v_1,\ldots,v_{\ell})$.
If a finite graph is viewed as a modeling with a uniform discrete probability measure on its vertex set,
then the Stone pairings for the graph and the modeling obtained in this way coincide.
A modeling $M$ is a \emph{modeling limit} of a first order convergent sequence $(G_n)_{n\in\NN}$
if 
$$\lim_{n\to\infty}\langle \psi,G_n\rangle=\langle\psi,M\rangle$$
for every first order formula $\psi$.

Every modeling limit $M$ of a first order convergent sequence of graphs satisfies
the \emph{finitary mass transport principle}. 
This means that
for any two given first order formulas $\psi$ and $\psi'$, each with one free variable, such that
every vertex $v$ satisfying $\psi(v)$ has at least $a$ neighbors satisfying $\psi'$ and
every vertex $v$ satisfying $\psi'(v)$ has at most $b$ neighbors satisfying $\psi$,
it holds that
\[a\langle\psi,M\rangle\le b\langle\psi',M\rangle\,\mbox{.}\]
For further details, we refer the reader to~\cite{NesO16}.

A stronger variant of this principle, known as the \emph{strong finitary mass transport principle}, requires that
the following holds for any measurable subsets $A$ and $B$ of the vertices of $M$:
if each vertex of $A$ has at least $a$ neighbors in $B$ and
each vertex of $B$ has at most $b$ neighbors in $A$, then
\[a\mu(A)\le b\mu(B)\,\mbox{,}\]
where $\mu$ is the probability measure of $M$.
Note that the assertion of the finitary mass transport principle requires this inequality to hold only
for first order definable subsets of vertices.

The strong finitary mass transport principle is satisfied by any finite graph
when viewed as a modeling but it need not hold for modelings in general;
for example, the modeling with the vertex set $[0,1]$ and
the edge set formed by a perfect matching between the Cantor set (or any other uncountable set of measure 0) and its complement
does not satisfy the strong finitary mass transport principle.
In particular, the existence of a modeling limit of a first order convergent sequence of graphs
does not a~priori imply the existence of a modeling limit satisfying the strong finitary mass transport principle, and
indeed the modeling limits constructed in~\cite{NesO19} do not satisfy the strong finitary mass transport principle in general.
The importance of the strong finitary mass transport principle comes from its relation to graphings,
which are limit representations of Benjamini-Schramm convergent sequences of bounded degree graphs, as
a modeling limit of a first order convergent sequence of bounded degree graphs 
must satisfy the strong finitary mass transport principle in order to be a graphing, see~\cite[Subsection~3.2]{NesO20}.

\subsection{Hintikka chains}
\label{subsec:Hintikka}

Our argument uses the topological space of Hintikka chains,
which was used in~\cite{GajHKKKOOT17} and which we now recall.
We remark that this space
is homeomorphic to the Stone dual of $\FOlocal$ studied in~\cite{NesO16,NesO19,NesO20} (the formal definition of $\FOlocal$ is below).
Hintikka sentences are maximally expressive sentences with a certain quantifier depth, and,
informally speaking, Hintikka chains form a local variant of this notion with a single free variable.

Consider a signature that includes the signature of graphs and
is finite except that it may contain countably many unary relational symbols.
In the exposition of our main argument,
the signature will consist of $k+2$ binary relational symbols and
countably many unary relation symbols:
$k$ binary relational symbols for $i$-edges for $i\in [k]$,
two  binary relational symbols for each of the two color classes of the $2$-edge-coloring, and
countably many unary relational symbols $U_i$, $i\in\NN$.
In our setting,
each unary relational symbols will be satisfied by at most one vertex and
will be used to pinpoint some special vertices in graphs that we consider.

A first order formula $\psi$
where each quantifier is restricted to the neighbors of one of the vertices
is said to be \emph{local},
i.e., it only contains quantifiers of the form $\forall_z x$ or $\exists_z x$ and
$x$ is required to be a neighbor of $z$ in one of the binary relations.
For example, we can express that a vertex $z$ has a neighbor of degree exactly one
with the local formula $\psi(z)\equiv\exists_z x\;\forall_x y\; y=z$.
We remark that a definition of local formulas that
permits quantifying over vertices in finite neighborhoods (rather than just neighbors)
does not yield a more general notion of local formulas;
indeed, quantifying over vertices in finite neighborhoods can be replaced by iterated quantifying over neighbors.
The quantifier depth of a local formula is defined in the usual way.
A formula is a \emph{$d$-formula} if its quantifier depth is at most $d$ and
it does not involve any unary predicates $U_i$ with $i> d$;
a $d$-formula with no free variables is referred to as a \emph{$d$-sentence}.

The same argument as in the textbook case of first order sentences yields that
there are only finitely many non-equivalent local $d$-formulas with one free variable for every $d$.
Let $\FOlocal$ be a maximal set of non-equivalent local formulas with one free variable,
i.e., a set containing one representative from each equivalence class of local formulas with one free variable, that
are chosen in a way that from each class a $d$-formula with $d$ as small as possible is chosen as a representative.
If $d$ is an integer, the \emph{$d$-Hintikka type} of a vertex $v$ of a (not necessarily finite) graph $G$
is the set of all $d$-formulas $\psi\in\FOlocal$ such that $G\models\psi(v)$.
We can visualize relations of mutually consistent Hintikka types by an infinite rooted tree $\TT_{\FOlocal}$:
every $d$-Hintikka types for $d\in\NN$ is associated with one of the non-root vertices of the tree $\TT_{\FOlocal}$,
every vertex associated with a $1$-Hintikka type is a child of the root, and
a vertex associated with a $d$-Hintikka type for $d\ge 2$
is a child of the unique vertex associated with the $(d-1)$-Hintikka type
that contains all formulas equivalent to the $(d-1)$-formulas contained in the $d$-Hintikka type.
Observe that every vertex of the infinite rooted tree $\TT_{\FOlocal}$ has a finite degree.

A $d$-formula $\psi\in\FOlocal$ is called a \emph{$d$-Hintikka formula}
if there exist a (not necessarily finite) graph $G$ and a vertex $v$ of $G$ such that
$\psi$ is equivalent to the conjunction of the formulas in the $d$-Hintikka type of the vertex $v$ of the graph $G$
(note that $\psi$ must actually be equivalent to one of the formulas in the $d$-Hintikka type of $v$);
in what follows,
we simply speak about a \emph{Hintikka formula} in case that $d$ is not important.
Fix a $d$-Hintikka formula $\psi$.
Observe that
if $v$ is a vertex of a graph $G$ and $v'$ is a vertex of a graph $G'$ such that $G\models\psi(v)$ and $G'\models\psi(v')$,
then the $d$-Hintikka types of $v$ and $v'$ are the same.
So, we can speak of the \emph{$d$-Hintikka type} of $\psi$;
these are all $d$-formulas $\psi'\in\FOlocal$ such that if $G\models\psi(v)$,
then $G\models\psi'(v)$ for any graph $G$ and any vertex $v$ of $G$.
Note that the $d$-Hintikka type of $\psi$ is uniquely determined by $\psi$, and
so each non-root vertex of the tree $\TT_{\FOlocal}$ can be associated with a single $d$-Hintikka formula.

A \emph{Hintikka chain} is a sequence $(\psi_d)_{d\in\NN}$ such that $\psi_d$ is a $d$-Hintikka formula and
the $d$-Hintikka type of $\psi_d$ contains formulas equivalent to $\psi_1,\ldots,\psi_{d-1}$.
In the infinite rooted tree $\TT_{\FOlocal}$,
Hintikka chains are in one-to-one correspondence with infinite paths from the root:
if $(\psi_d)_{d\in\NN}$ is a Hintikka chain,
then the path is formed by the root and the vertices associated with the $d$-Hintikka type of $\psi_d$.
This implies that for every $d$-Hintikka formula $\psi_d$,
there are only finitely many $(d+1)$-Hintikka formula $\psi_{d+1}$ such that
the $(d+1)$-Hintikka type of $\psi_{d+1}$ contains $\psi_d$.
Following the standard terminology related to infinite rooted trees,
we refer to infinite paths from the root in the tree $\TT_{\FOlocal}$ as to $\emph{rays}$, and
when we speak of a \emph{subtree} of $\TT_{\FOlocal}$, we mean a subgraph formed by a vertex and all its descendants.
Note that if $G$ is a finite graph and $v$ a vertex of $G$,
then the Hintikka chain $(\psi_d)_{d\in\NN}$ such that $G\models\psi_d(v)$ satisfies that
there exists $d_0\in\NN$ such that $\psi_d=\psi_{d_0}$ for all $d\ge d_0$.
However, this is not true for infinite graphs $G$.

If $\psi$ is a $d$-Hintikka formula,
the set $\{(\psi_i)_{i\in\NN}:\psi_d=\psi\}$ of Hintikka chains is called \emph{basic}.
Observe that basic sets of Hintikka chains correspond to sets of rays in $\TT_{\FOlocal}$ that
lead to the same subtree of the tree $\TT_{\FOlocal}$.
It is a classical fact in descriptive set theory, see e.g.~\cite[Chapter 2]{Kec95} or~\cite[Section 2.B]{Tse19},
that rays of an infinite rooted tree with vertices of finite degrees with the topology generated
by clopen sets of rays passing through the same vertex form a Polish space.
In particular,
the set of Hintikka chains (formed by Hintikka formulas with a fixed signature) equipped
with the topology with the base formed by basic sets is a Polish space,
which is denoted by $\SSS_{\FOlocal}$ in the rest of the paper.
We remark that the Polish space $\SSS_{\FOlocal}$
is homeomorphic to the Stone dual of $\FOlocal$ studied in~\cite{NesO20,NesO12}.
Finally,
observe that for every formula $\psi\in\FOlocal$, say $\psi$ is a $d$-formula,
the set of all Hintikka chains $(\psi_i)_{i\in\NN}$ such that the $d$-Hintikka type of $\psi_{d}$ contains $\psi$
is a finite union of basic sets.
In particular, the set of all such Hintikka chains is a clopen set in $\SSS_{\FOlocal}$.

\section{Decomposition step}
\label{sec:decompose}

Our main argument consists of two steps, which we refer in the analogy to~\cite{GajHKKKOOT17}
as the decomposition step and the composition step.
In this section, we present the former,
which is analogous to that for plane trees given in~\cite[Section 3.1]{GajHKKKOOT17};
this step is also closely related to comb structure results presented in~\cite{NesO20,NesO16}.
As the decomposition step follows closely \cite[Section 3.1]{GajHKKKOOT17} and the arguments are analogous,
we present the main ideas and refer the reader to~\cite[Section 3.1]{GajHKKKOOT17} for further details.

We start with recalling the following structural result from~\cite{NesO16c}.

\begin{theorem}
\label{thm:major}
Let $\CC$ be a nowhere-dense class of graphs.
For every $d\in\NN$ and $\varepsilon>0$,
there exists $N\in\NN$ such that every graph $G\in\CC$ contains a set $S$ of at most $N$ vertices such that
the $d$-neighborhood of every vertex $G\setminus S$ contains at most $\varepsilon |G|$ vertices.
\end{theorem}

We will enhance the signature of graphs by countably many unary relations $U_i$, $i\in\NN$;
these relations will identify vertices described in Theorem~\ref{thm:major},
whose removal makes the considered sequence of graphs residual in the sense of~\cite{NesO16,NesO19}.
In particular, we will require that no vertex satisfies more than one of the relations $U_i$, $i\in\NN$, and
each relation is satisfied by at most one vertex.
To simplify our notation, we will write $c_i$ for the vertex satisfying $U_i$ if it exists.
We will refer to graphs with unary relations $U_i$, $i\in\NN$, as
to \emph{marked graphs} following the terminology used in~\cite{NesO19}.
A sequence of marked graphs $(G_n)_{n\in\NN}$ is {\em null-partitioned}
if the following two conditions hold (the definition is analogous to that for plane trees in~\cite{GajHKKKOOT17}):
\begin{itemize}
\item for every $k\in\NN$,
      there exists $n_k\in\NN$ such that
      there exist unique distinct vertices satisfying $U_1,\ldots,U_k$ in $G_n$ for every $n\ge n_k$,
      i.e., the vertices $c_1,\ldots,c_k$ are well-defined in $G_n$ for every $n\ge n_k$, and
\item for every $\varepsilon>0$, there exist integers $k_0$ and $n_0\ge n_{k_0}$ such that
      the size of each component of $G_n\setminus\{c_1,\ldots,c_{k_0}\}$ is at most $\varepsilon |G_n|$ for every $n\ge n_0$ (note
      that the vertices satisfying $U_1,\ldots,U_{k_0}$ in $G_n$ exist since $n_0\ge n_{k_0}$).
\end{itemize}
Note that if $(G_n)_{n\in\NN}$ is null-partitioned,
then the number of vertices of the graphs $G_n$ must tend to infinity (otherwise,
the first property cannot hold).

The following lemma, which is a counterpart of~\cite[Lemma 4]{GajHKKKOOT17} and the proof follows the same lines,
readily follows from Theorem~\ref{thm:major}, and
also follows from the existence of marked quasi-residual sequences used in~\cite{NesO19}.

\begin{lemma}
\label{lm:unary}
Let $\CC$ be a nowhere-dense monotone class of graphs.
Let $(G_n)_{n\in\NN}$ be a first order convergent sequence of graphs $G_n\in\CC$ such that the orders of $G_n$ tend to infinity.
There exists a first order convergent null-partitioned sequence $(G'_n)_{n\in\NN}$
obtained from a subsequence of $(G_n)_{n\in\NN}$ by interpreting unary relational symbols $U_i$, $i\in\NN$.
\end{lemma}

The first order properties are linked with Ehrenfeucht-Fra{\"\i}ss\'e games~\cite{EbbF05}.
It is well-known that two structures satisfy the same first order sentences with quantifier depth $d$ if and only if
the duplicator has a winning strategy for the $d$-round Ehrenfeucht-Fra{\"\i}ss\'e game played on two structures,
in our settings, on two $2$-edge-colored rooted marked $k$-trees.
The $d$-round Ehrenfeucht-Fra{\"\i}ss\'e game is played by two players, called the spoiler and the duplicator.
In the $i$-th round, the spoiler chooses a vertex of one of the $k$-trees (the spoiler can choose a different $k$-tree
in different rounds) and places the $i$-th pebble on that vertex.
The duplicator responds with placing the $i$-th pebble on a vertex of the other $k$-tree.
At the end of the game,
the duplicator wins
if the mapping between the subforests induced by the vertices with the pebbles that
maps the vertex with the $i$-pebble in the first $k$-tree to the vertex with the $i$-pebble in the other $k$-tree
is an isomorphism preserving the $j$-parent relations, $j\in [k]$, the unary relations $U_i$, $i\in\NN$, and the edge colors.
If two $2$-edge-colored rooted marked $k$-trees $G$ and $G'$ satisfy the same $d$-sentences if and only if
the duplicator has a winning strategy for the $d$-round Ehrenfeucht-Fra{\"\i}ss\'e game
when played with two $2$-edge-colored rooted marked $k$-trees $G$ and $G'$
with all unary relations $U_i$, $i>d$, set to be empty.

The following theorem is the counterpart of~\cite[Theorem 6]{GajHKKKOOT17}.
The theorem can be proven using an argument analogous to that used to prove \cite[Lemma 5]{GajHKKKOOT17},
which is based on the link of $d$-sentences to the $d$-round Ehrenfeucht-Fra{\"\i}ss\'e games presented above.
However, in our (simpler) setting not involving ordered edge incidences,
the statement of the theorem for each $d\in\NN$
directly follows from Hanf's Theorem applied to relational structures obtained by removing unary relations $U_i$ with $i>d$.
An analogous argument yields that Gaifman's Locality Theorem also extends to our setting
when neighborhoods of vertices are replaced with Hintikka types of vertices.

\begin{theorem}
\label{thm:hanf}
For every integer $d$, there exist integers $D$ and $\Gamma$ such that
for any two (not necessarily finite) $2$-edge-colored rooted marked $k$-trees $G$ and $G'$
if the $k$-trees $G$ and $G'$ have the same number of vertices of each $D$-Hintikka type or
the number of the vertices of this type is at least $\Gamma$ in both $G$ and $G'$,
then the sets of $d$-sentences satisfied by $G$ and $G'$ are the same.
\end{theorem}

Fix a first order convergent null-partitioned sequence of $2$-edge-colored rooted marked $k$-trees $(G_n)_{n\in\NN}$.
We associate the sequence with two functions, $\nu:\FOlocal\to\NN_0\cup\{\infty\}$ and $\mu:\FOlocal\to [0,1]$,
which we will refer to as the \emph{discrete Stone measure} and the \emph{Stone measure} of the sequence (strictly speaking,
$\nu$ and $\mu$ are not measures as they are not defined on a $\sigma$-algebra,
however, each can be extended to a measure on the $\sigma$-algebra formed by Borel sets of $\SSS_{\FOlocal}$
as we will argue in the next paragraph).
If $\psi\in\FOlocal$, then $\nu(\psi)$ is the limit of the number of vertices $u$ such that $G_n\models\psi(u)$, and
$\mu(\psi)$ is the limit of $\langle\psi,G_n\rangle$.
In the analogy to the finite case,
if $M$ is a modeling,
we define its \emph{discrete Stone measure} and its \emph{Stone measure}
by setting $\nu(\psi)=|\{u\mbox{ s.t. }M\models\psi(u)\}|$ and $\mu(\psi)=\langle\psi,M\rangle$ for $\psi\in\FOlocal$.
Again, $\nu$ and $\mu$ are not measures in the sense of measure theory as
they are mappings from $\FOlocal$ and not from a $\sigma$-algebra.

As we have already mentioned, the Stone measure $\mu$ defined in the previous paragraph
yields a measure on the $\sigma$-algebra formed by Borel sets of the topological space $\SSS_{\FOlocal}$,
whose elements (points) are Hintikka chains.
This follows from~\cite[Theorem 2.6]{NesO20} but we sketch a short self-contained argument.
Let $\RR$ be the ring formed by finite unions of basic sets of Hintikka chains;
observe that $\RR$ consists of finite unions of disjoint basic sets of Hintikka chains,
which were introduced in Subsection~\ref{subsec:Hintikka}, and
each basic set of Hintikka chains is formed by all Hintikka chains containing a particular Hintikka formula.
Let $X\in\RR$ be the union of disjoint basic sets corresponding to Hintikka formulas $\psi_1,\ldots,\psi_k$ and
define $\mu_{\RR}(X)=\mu(\psi_1)+\cdots+\mu(\psi_k)$.
Clearly, the mapping $\mu_{\RR}$ is additive.
Since every countable union of non-empty pairwise disjoint sets from $\RR$ that is contained in $\RR$
must be finite (this is implied by K\"onig's Lemma applied to the infinite rooted tree $\TT_{\FOlocal}$),
the mapping $\mu_{\RR}$ is a premeasure.
By Carath\'eodory's Extension Theorem, the premeasure $\mu_{\RR}$ extends to a measure
on the $\sigma$-algebra formed by Borel sets of Hintikka chains.
We will also use $\mu$ to denote this measure,
which is uniquely determined by the Stone measure $\mu$;
we believe that no confusion can arise by doing us so.

The following lemma relates first order convergent sequences of $2$-edge-colored rooted marked $k$-trees and their modeling limits.
The proof is analogous to that of~\cite[Lemma~7]{GajHKKKOOT17}, and
it can also be derived from~\cite[Theorem 36]{NesO16}, also see~\cite[Proof of Theorem~17]{NesO19}.

\begin{lemma}
\label{lm:modeling}
Let $(G_n)_{n\in\NN}$ be a first order convergent null-partitioned sequence of $2$-edge-colored rooted marked $k$-trees with increasing orders and
let $\nu$ and $\mu$ be its discrete Stone measure and Stone measure, respectively.
If $M$ is a $2$-edge-colored rooted $k$-tree modeling such that
\begin{itemize}
\item the discrete Stone measure of $M$ is $\nu$,
      in particular, there exists exactly one vertex $c_i$ of $M$ that satisfies the unary relation $U_i$,
\item the Stone measure of $M$ is $\mu$,
\item the $r$-neighborhood in $M\setminus\{c_i,i\in\NN\}$ of each vertex in $M\setminus\{c_i,i\in\NN\}$
      has zero measure for every $r\in\NN$, 
\end{itemize}
then $M$ is a modeling limit of $(G_n)_{n\in\NN}$.
\end{lemma}

Note that Lemma~\ref{lm:modeling} says in particular that
if $M$ is a modeling such that
\[\lim_{n\to\infty}\langle \psi,G_n\rangle=\langle\psi,M\rangle\]
for every $\psi\in\FOlocal$,
the discrete Stone measure of $M$ is $\nu$, and
the modeling has the residuality property in the sense of~\cite{NesO16,NesO19}
with respect to the vertices that null-partition graphs $G_n$ (this is the third property in Lemma~\ref{lm:modeling}),
then
\[\lim_{n\to\infty}\langle \psi,G_n\rangle=\langle\psi,M\rangle\]
for every first order formula $\psi$.

\section{Composition step}
\label{sec:compose}

The following theorem is the core result of this paper.

\begin{theorem}
\label{thm:compose}
Fix a positive integer $k$.
Let $(G_n)_{n\in\NN}$ be a first order convergent null-partitioned sequence of $2$-edge-colored rooted marked $k$-trees and
let $\nu$ and $\mu$ be its discrete Stone measure and Stone measure, respectively.
Then there exists a modeling $M$ such that
\begin{itemize}
\item the discrete Stone measure of $M$ is $\nu$,
\item the Stone measure of $M$ is $\mu$,
\item there exists exactly one vertex $c_i$ of $M$ that satisfies the unary relation $U_i$ for all $i\in\NN$,
\item the $r$-neighborhood of each vertex in $M\setminus\{c_i,i\in\NN\}$ has zero measure for every $r\in\NN$,
\item the modeling $M$ satisfies the strong finitary mass transport principle, and
\item the modeling $M$ also satisfies the strong finitary mass transport principle
      when all edges of one of the two colors are removed.
\end{itemize}
\end{theorem}

\begin{proof}
We start by extending the mapping $\nu$ to Hintikka chains, i.e. elements of the topological space $\SSS_{\FOlocal}$:
if $\Psi=(\psi_i)_{i\in\NN}$ is a Hintikka chain, then we define
\[\nu(\Psi)=\lim_{i\to\infty}\nu(\psi_i)\,\mbox{;}\]
note that the limit always exists since $\nu(\psi_i)_{i\in\NN}$ is non-increasing.
Observe that the set $\nu^{-1}(\infty)$ is closed in $\SSS_{\FOlocal}$.
Indeed, the set $\nu^{-1}(\infty)$
is the complement of the countable union of the basic sets of Hintikka chains
corresponding to Hintikka formulas $\psi$ with $\nu(\psi)\in\NN_0$.

Observe that the support of the measure $\mu$ on the $\sigma$-algebra of Borel sets of $\SSS_{\FOlocal}$,
which is yielded by the Stone measure $\mu$,
is the set containing all Hintikka chains $(\psi_i)_{i\in\NN}$
such that the measure of the basic set of Hintikka chains corresponding to $\psi_i$ is positive for every $i\in\NN$.
We next show that the support of the measure $\mu$ is a subset of $\nu^{-1}(\infty)$;
in general, the support of $\mu$ need not be equal to $\nu^{-1}(\infty)$.
Assume for contradiction that $\nu(\Psi)=k\in\NN$ for some $\Psi$ in the support of $\mu$.
Then there exists a $d$-Hintikka formula $\psi_d$ such that $\nu(\psi_d)=k$.
Hence, there exists $n_0$ such that every $G_n$, $n\ge n_0$,
contains exactly $k$ vertices $u_1,\dots,u_k$ such that $G_n\models\psi_d(u_i)$ for $i\in[k]$.
Since the number of vertices of the graphs $G_n$, $n\in\NN$ tend to infinity (as
the sequence $(G_n)_{n\in\NN}$ is null-partitioned) and
the Stone pairing $\langle \psi_d,G_n\rangle$ for every $n\ge n_0$
is equal to $k$ divided by the number of vertices of $G_n$,
it follows that $\mu(\psi_d)=0$.
Consequently,
the measure of the basic set of Hintikka chains corresponding to $\psi_d$ is zero
contradicting the assumption that $\Psi$ is in the support of $\mu$.

The rest of the proof of the theorem is split into several steps,
each of which we start with a brief title that gives a description of the step.

\textbf{Vertex set.} The vertex set of the modeling $M$ that we construct consists of two sets:
$V_f$ contains all pairs $(\Psi,i)$ where $\Psi$ is a Hintikka chain such that $\nu(\Psi)\in\NN$ and $i\in [\nu(\Psi)]$, and
$V_{\infty}=\nu^{-1}(\infty)\times [0,1)^{2k+1}$.
Observe that the set $V_f$ is countable;
indeed, if $\nu(\Psi)\in\NN$,
then there exists $d\in\NN$ such that the $d$-th formula $\psi_d$ of $\Psi$ satisfies that $\nu(\psi_d)=\nu(\Psi)$;
since the number of Hintikka formulas is countable, it follows that $V_f$ is countable.
Also note that for every $i\in\NN$, the set $V_f$ contains the unique vertex satisfying $U_i$.
The topology on $V_f\cup V_{\infty}$ is defined later.

A Hintikka chain encodes many properties of a vertex.
In particular, it uniquely determines the Hintikka chain of the $i$-parent of the vertex $v$ for each $i\in [k]$,
which we will refer to as the $i$-parent Hintikka chain.
The Hintikka chain also determines whether the vertex is one of the vertices satisfying $U_i$, $i\in\NN$, and
whether it is contained in the initial tournament (initial stands here for initial in the recursive definition of a rooted $k$-tree).
Furthermore,
the Hintikka chain also determines the number of children that satisfy a particular local first order formula and,
more generally, that have a certain Hintikka chain.
Finally,
the Hintikka chain describes the existence or the non-existence of finite paths consisting only of $A$-edges for arbitrary $A\subseteq [k]$ between the parents of the vertex.
We say that an $i$-edge $e$ contained in such a path is \emph{$m$-finitary} for $m\in\NN$
if the head of $e$ is the head of exactly $m$ $i$-edges whose tail has the same Hintikka chain as the tail of $e$.
If an edge $e$ is $m$-finitary for some $m$, we say that $e$ is \emph{finitary}, otherwise
if $e$ is not $m$-finitary for any $m$, then $e$ is \emph{infinitary}.
Note that whether $e$ is $m$-finitary for some $m\in\NN$ or infinitary
is implied by the Hintikka chain of either the head or the tail of $e$.
In a slightly informal way,
we will be speaking about all these properties as the properties of the Hintikka chain.

\textbf{Unary relations.}
We now continue with the construction of the modeling $M$ with defining the unary relations $U_i$, $i\in\NN$.
For each $i\in\NN$,
there is exactly one Hintikka chain $\Psi$ with $\nu(\Psi)>0$ such that the vertex of $\Psi$ satisfies $U_i$.
It holds $\nu(\Psi)=1$ for such a Hintikka chain $\Psi$, and
the vertex $(\Psi,1)\in V_f$ will be the unique vertex of $M$ satisfying the unary relation $U_i$.
In what follows, we write $c_i$ for this vertex of $M$.

\textbf{Edges.}
We next define the edges of the modeling $M$
by describing the edges leading from each vertex of $M$ to its parents.
To do so,
we fix for every $d\in [2k]$ a continuous measure preserving bijection $\zeta^d:[0,1)\to [0,1)^d$;
for $x\in [0,1)$ and $i\in [d]$, we write $\zeta^d_i(x)$ for the $i$-th coordinate of $\zeta^d(x)$.

We first consider the vertices contained in $V_f$.
Let $(\Psi,m)\in V_f$ (note that such a vertex may be one of the vertices $c_i$, $i\in\NN$).
For every $i\in [k]$ such that $\Psi$ implies the existence of an $i$-parent, we proceed as follows.
The definition of the discrete Stone measure and the first order convergence of $(G_n)_{n\in\NN}$ imply that
there exists $d_0\in\NN$ such that
for every $d\ge d_0$
there exists $n_d\in\NN$ such that
every graph $G_n$, $n\ge n_d$, has exactly $\nu(\Psi)$ vertices satisfying the $d$-th formula of the Hintikka chain $\Psi$.
Let $\Psi'$ be the $i$-parent Hintikka chain of $\Psi$.
Observe that $\Psi'\not=\Psi$;
otherwise,
since
the $i$-parent of every vertex of $G_{n_{d_0+1}}$ whose $(d_0+1)$-Hintikka type is the $(d_0+1)$-th formula of the Hintikka chain $\Psi$
has the $d_0$-Hintikka type equal to the $d_0$-th formula of the Hintikka chain $\Psi$,
the $(d_0+1)$-Hintikka type of the $i$-parent is equal to the $(d_0+1)$-th formula of the Hintikka chain $\Psi$, and
so the vertices whose $(d_0+1)$-Hintikka type is the $(d_0+1)$-th formula of the Hintikka chain $\Psi$ form a cycle,
which is impossible since any rooted $k$-tree is acyclic.
Note that the Hintikka chain of $\Psi'$ determines the number of $i$-children
whose $d_0$-Hintikka type is the $d_0$-th formula of the Hintikka chain $\Psi$, and
let $D$ be this number.
Since every graph $G_n$, $n\ge n_{d+D+2}$, $d\ge d_0$,
has exactly $\nu(\Psi)$ vertices whose $d$-Hintikka type is the $d$-th formula of the Hintikka chain $\Psi$,
the $i$-parent of each of these vertices have its $(d+D+1)$-Hintikka type equal to the $(d+D+1)$-th formula of the Hintikka chain $\Psi'$, and
this $i$-parent has exactly $D$ $i$-children whose $d$-Hintikka type is the $d$-th formula of the Hintikka chain $\Psi$,
we obtain that $\nu(\Psi')\not=\infty$ and $\nu(\Psi)=D\nu(\Psi')$.
In particular, $\nu(\Psi')$ is a non-zero integer which divides $\nu(\Psi)$.
In the modeling $M$,
we set the $i$-parent of the vertex $(\Psi,m)$ to be the vertex $(\Psi',m')$ where $m'=\lceil m\nu(\Psi')/\nu(\Psi) \rceil$, and
the color of the edge from $(\Psi,m)$ to $(\Psi',m')$ as determined by the Hintikka chain $\Psi$.
This determines all edges between the vertices of $V_f$,
which include all edges leaving the vertices of $V_f$.

Next, consider a vertex $(\Psi,n_0,h_1,n_1,\ldots,h_k,n_k)\in V_{\infty}$.
Fix $i\in [k]$.
Since there are at most $k$ vertices without an $i$-parent (these are the vertices of the initial tournament
in the definition of a rooted $k$-tree),
the Hintikka chain $\Psi$ determines that the vertex described by $\Psi$ has an $i'$-parent for every $i'\in [k]$,
in particular, it has an $i$-parent.
Let $\Psi'$ be the Hintikka chain of the $i$-parent.

We consider all finite directed paths from $(\Psi,n_0,h_1,n_1,\ldots,h_k,n_k)$ to its $i$-parent whose existence is implied by $\Psi$;
we will say that a directed path (whose existence is implied by $\Psi$) from the tail to the head of an $m$-edge $e$
is a \emph{detour} of $e$ if
\begin{itemize}
\item $e$ is a finitary edge and the path is formed by $[m-1]$-edges only, or
\item $e$ is an infinitary edge and all infinitary edges of the path are $[m-1]$-edges.
\end{itemize}
An edge with no detour is said to be \emph{important}.
Observe that if $e$ is a finitary edge,
then any detour of $e$ is formed by finitary edges only (if a detour contained an infinitary edge,
then there would be infinitely many directed paths from vertices with the same Hintikka chain as the tail of $e$ to the head of $e$ and
since such paths start at different vertices, $e$ would be infinitary).
Observe that
$\Psi$ implies the existence of a directed path from the vertex $(\Psi,n_0,h_1,n_1,\ldots,h_k,n_k)$ to its $i$-parent that
is formed by important edges only.
Indeed, replacing an $i'$-edge $e$ that is not important with a detour of $e$
either decreases the number of infinitary $\{i',\ldots,k\}$-edges, or
preserves the number of infinitary $\{i',\ldots,k\}$-edges while decreasing the number of all $\{i',\ldots,k\}$-edges;
hence, the process of replacing edges that are not important with corresponding detours as long as possible always terminates.

Fix a directed path $P$ from $(\Psi,n_0,h_1,n_1,\ldots,h_k,n_k)$ to its $i$-parent formed by important edges only;
note that $P$ is a type of a directed path rather than an actual path.
Let $\ell$ be the length of the path $P$.
Set $\Psi^0$ to be the Hintikka chain $\Psi$ and
set $\Psi^j$ for $j=1,\ldots,\ell$ to be the Hintikka chain of the head of the $j$-th edge of the path $P$.
Let $\ell_\infty$ the largest $j$ such that $\nu(\Psi^j)=\infty$ (possibly $\ell_\infty=0$ or $\ell_\infty=\ell$).

We next identify vertices $(\Psi^j,n^j_0,h^j_1,n^j_1,\ldots,h^j_k,n^j_k)$ for $j=1,\ldots,\ell_{\infty}$,
which are actual vertices in a modeling $M$;
the vertex $(\Psi^j,n^j_0,h^j_1,n^j_1,\ldots,h^j_k,n^j_k)$ should be the head of the $j$-th important edge on the actual path
from the vertex $(\Psi,h_1,n_1,\ldots,h_k,n_k)$ to its $i$-parent in the modeling $M$ that corresponds to $P$.
So, we define $(2k+1)$-tuples $(n^j_0,h^j_1,n^j_1,\ldots,h^j_k,n^j_k)$ for $j=0,\ldots,\ell_{\infty}$ recursively as follows.

We first set $(n^0_0,h^0_1,n^0_1,\ldots,h^0_k,n^0_k)$ to $(n_0,h_1,n_1,\ldots,h_k,n_k)$.
If the $j$-th edge of the path $P$ is an infinitary $i'$-edge,
we set
\begin{align*}
(n^j_0,h^j_1,n^j_1,\ldots,h^j_k,n^j_k)=(&\zeta_1^{2i'}(n^{j-1}_{i'}),\ldots,\zeta^{2i'}_{2i'-1}(n^{j-1}_{i'}),\\
                                        &h^{j-1}_{i'}+\sqrt{2}\mod 1,\zeta_{2i'}^{2i'}(n^{j-1}_{i'}),\\
	                                &h^{j-1}_{i'+1},n^{j-1}_{i'+1},\ldots,h^{j-1}_k,n^{j-1}_k).
\end{align*}
If the $j$-th edge of the path $P$ is an $m$-finitary $i'$-edge and
the edge from the head of the $j$-th edge to its $i''$-parent for every $i''=i'+1,\ldots,k$
is infinitary or has a detour,
we set
\begin{align*}
(n^j_0,h^j_1,n^j_1,\ldots,h^j_k,n^j_k)=(&m\cdot n^{j-1}_0\mod 1,h^{j-1}_1,n^{j-1}_1,\ldots,h^{j-1}_{i'-1},n^{j-1}_{i'-1},\\
                                        &h^{j-1}_{i'}+\sqrt{2}\mod 1,n^{j-1}_{i'},\\
					&h^{j-1}_{i'+1},n^{j-1}_{i'+1},\ldots,h^{j-1}_k,n^{j-1}_k).
\end{align*}
Otherwise,
i.e., when the $j$-th edge of the path $P$ is an $m$-finitary $i'$-edge and
there exists $i''>i'$ such that the edge from the head of the $j$-th edge to its $i''$-parent is finitary and has no detour,
we set
\begin{align*}
(n^j_0,h^j_1,n^j_1,\ldots,h^j_k,n^j_k)=(m\cdot n^{j-1}_0\mod 1,h^{j-1}_1,n^{j-1}_1,\ldots,h^{j-1}_k,n^{j-1}_k).
\end{align*}

If $\ell=\ell_\infty$,
set the $i$-parent of $(\Psi,n_0,h_1,n_1,\ldots,h_k,n_k)$ to be $(\Psi',n^\ell_0,h^\ell_1,n^\ell_1,\ldots,h^\ell_k,n^\ell_k)$ and
the color of the edge from $(\Psi,n_0,h_1,n_1,\ldots,h_k,n_k)$ to the $i$-parent as determined by $\Psi$;
observe that the $i$-parent of $(\Psi,n_0,h_1,n_1,\ldots,h_k,n_k)$ is always different from $(\Psi,n_0,h_1,n_1,\ldots,h_k,n_k)$ as
$h^\ell_j=h_j+m\sqrt{2}\mod 1$ for some $m\in\NN$ where $j$ is the largest index such that $h_j\not=h^\ell_j$ (note that
it is possible that $\Psi=\Psi'$).

If $\ell>\ell_\infty$,
we set the $i$-parent of $(\Psi,n_0,h_1,n_1,\ldots,h_k,n_k)$ to be $\left(\Psi',1+\lfloor\nu(\Psi')\cdot n^{\ell_\infty}_k\rfloor\right)$;
the color of the edge from $(\Psi,n_0,h_1,n_1,\ldots,h_k,n_k)$ to the $i$-parent is again determined by $\Psi$.
Note that
in the latter case
we can think of the $j$-th vertex for $j>\ell_\infty$ as being $\left(\Psi^j,1+\lfloor\nu(\Psi^j)\cdot n^{\ell_\infty}_k\rfloor\right)$,
which is consistent with the definition of edges among the vertices of $V_f$.
This concludes the definition of the edge set of the modeling $M$.

\textbf{Well-defined.}
We next verify that the definition of the $i$-parent of a vertex does not depend on the choice of a directed path $P$ (as long as
$P$ is formed by important edges only).
We prove that
if $P$ and $P'$ are two directed paths from a vertex $(\Psi,n_0,h_1,n_1,\ldots,h_k,n_k)$ to its $i$-parent that
are formed by important edges only,
then they yield the same definition of the $i$-parent;
the proof proceeds by induction on the sum of the lengths of $P$ and $P'$.
The base of the induction is formed by the case
when the sum of the lengths of $P$ and $P'$ is equal to two,
i.e., the directed paths $P$ and $P'$ are actually the same edge and so the claim holds.

We now present the induction step.
Fix a vertex $v=(\Psi,n_0,h_1,n_1,\ldots,h_k,n_k)$ and
two directed paths $P$ and $P'$ from the vertex $v$ to its $i$-parent, which we denote by $u$, that
both are formed by important edges only.
Let $w$ and $w'$ be the heads of the first edge of $P$ and $P'$, and
let $j$ and $j'$ be such that the first edges of $P$ and $P'$ are an $j$-edge and an $j'$-edge, respectively.
See Figure~\ref{fig:welldefined} for the notation.
By symmetry, we may assume that $j\le j'$.
Note that $w$ is the $j$-parent of $v$, $w'$ is its $j'$-parent, 
$u$ is the $i'$-parent of $w$ for some $i'\leq i$, and
let $\Psi_w$ and $\Psi_{w'}$ be the Hintikka chains of $w$ and $w'$, respectively.

We first deal with the case that $j=j'$, which implies that $w=w'$.
Consider the directed paths $Q$ and $Q'$ obtained from $P$ and $P'$ by removing their first edge, respectively.
If $\nu(\Psi_w)=\infty$,
the induction assumption implies that the paths $Q$ and $Q'$ yield the same definition of the $i'$-parent of $w$.
If $\nu(\Psi_w)\not=\infty$,
the definition of edges between the vertices of $V_f$ implies that
the paths $Q$ and $Q'$ yield the same definition of the $i'$-parent of $w$.
In either of the cases, the paths $P$ and $P'$ yield the same definition of the $i$-parent of $v$.

\begin{figure}
\begin{center}
\epsfbox{folim-tw-1.mps}
\end{center}
\caption{Notation used in the proof of Theorem~\ref{thm:compose} when establishing that edges of $M$ are well-defined.}
\label{fig:welldefined}
\end{figure}

We next deal with the case that $j\not=j'$.
By Proposition~\ref{prop:parent} applied rooted $k$-trees $G_n$,
the Hintikka chain $\Psi$ implies that the $j'$-parent of $v$ is
the $j''$-parent of the the $j$-parent of $v$ for some $j''\le j'$,
i.e., $w'$ should be the $j''$-parent of $w$;
note that we still need to establish that
the vertex defined to be the $j'$-parent of $v$ by the first edge of the path $P'$ in $M$
is also the $j''$-parent of the $j$-parent of $v$ defined by the first edge of the path $P$, and
this is indeed the core of the argument when $\nu(\Psi_w)=\infty$.
Since the first edge of $P'$ is important,
it holds that $j'=j''$ (otherwise, $vww'$ is a detour of the edge $vw'$) and
the edge $ww'$ is important (otherwise, the edge $vw$ together with a detour of the edge $ww'$ is a detour of the edge $vw'$).
Let $Q$ be the directed path $P$ with the first edge removed and $Q'$ be the directed path $P'$ with the first edge replaced with $ww'$.

If $\nu(\Psi_w)\not=\infty$,
then $\nu(\Psi_{w'})\not=\infty$ and
the definition of the edges from $V_{\infty}$ to $V_f$ yields that
$w=(\Psi_w,1+\lfloor\nu(\Psi_w)n_k\rfloor)$ and $w'=(\Psi_{w'},1+\lfloor\nu(\Psi_{w'})n_k\rfloor)$ (the latter
is the same regardless whether we use the edge $vw'$ to define $w'$ or the edge $ww'$ once $w$ has been defined).
The definition of edges between the vertices of $V_f$ implies that
the paths $Q$ and $Q'$ yield the same definition of the $i'$-parent of $w$.

If $\nu(\Psi_w)=\infty$,
the paths $Q$ and $Q'$ yield the same definition of the $i'$-parent of $w$ by induction as
they both are formed by important edges only.
To finish the proof of the induction step,
we need to establish that the paths $vw'$ and $vww'$ yield the same definition of the vertex $w'$,
which is the $j'$-parent of $v$.
We first consider the case that $\nu(\Psi_{w'})\not=\infty$.
It follows that both edges $vw'$ and $ww'$ are important infinitary $j'$-edges,
which in turn implies that the last coordinate of $v$ and $w$ is the same.
Hence, $w'$ is $(\Psi_{w'},1+\lfloor\nu(\Psi_{w'})n_k\rfloor)$
when defined using either of the paths $vw'$ and $vww'$.
In the rest, we assume that both $\nu(\Psi_w)$ and $\nu(\Psi_{w'})$ are equal to $\infty$.

If the edge $vw'$ is infinitary, then the edge $ww'$ is also infinitary (otherwise, $vww'$ is a detour of $vw'$).
As both edges $vw'$ and $ww'$ are important infinitary $j'$-edges and $vw$ is an $j$-edge for $j<j'$,
the vertex $w'$ is the same when defined using either of the paths $vw'$ and $vww'$.
If the edge $vw'$ is $m$-finitary, then both edges $vw$ and $ww'$ must also be finitary;
moreover, if $vw$ is $m_1$-finitary and $ww'$ is $m_2$-finitary, then $m=m_1m_2$.
Hence, the vertex $w$ is $(\Psi_w,n_0\cdot m_1\mod 1,h_1,n_1,\ldots,h_k,n_k)$ as
the edge $vw$ is a $j$-edge and the edge $ww'$ is an important $j'$-edge for $j'>j$.
If there exists an edge from $w'$ to its $\widehat{j}$-parent for some $\widehat{j}>j'$ that is finitary and has no detour,
then the vertex $w'$ is $(\Psi_{w'},n_0\cdot m\mod 1,h_1,n_1,\ldots,h_k,n_k)$.
Otherwise, the vertex $w'$ is $(\Psi_{w'},n_0\cdot m\mod 1,h_1,\ldots,n_{j'-1},h_{j'}+\sqrt{2}\mod 1,n_{j'},\ldots,n_k)$.
Hence, the vertex $w'$ is the same when defined using either of the paths $vw'$ and $vww'$ in all cases.

Since the edges of $M$ are well-defined,
it follows that the number of $i$-children with a given Hintikka type of every vertex $v$ of $M$
is the number determined by the Hintikka type of $v$.

\textbf{Edge consistency.}
We next verify the following property, which we refer to as \emph{edge consistency}.
For every vertex $w$, the following holds:
if $w'$ and $w''$ are the $i'$-parent and the $i''$-parent of $w$, $i'<i''$, respectively, and
$w''$ is the $i$-parent of $w'$ according to the Hintikka chain of $w$,
then $w''$ is the $i$-parent of $w'$ in $M$ (the notation is illustrated in Figure~\ref{fig:edge-www}).
Note that this property is not automatically satisfied as the definition of edges from $w'$ to its parents
is independent of the definition of edges from $w$ to its parents.

\begin{figure}
\begin{center}
\epsfbox{folim-tw-2.mps}
\end{center}
\caption{Notation used in the proof of Theorem~\ref{thm:compose} when establishing that the edge consistency.}
\label{fig:edge-www}
\end{figure}

If $w\in V_f$,
then the edge consistency straightforwardly follows from the definition of edges among the vertices of $V_f$.
If $w\in V_{\infty}$, consider a directed path $P'$ from $w$ to $w'$ and a directed path $P''$ from $w'$ to its $i$-parent,
each formed by important edges only.
Note that the concatenation of the paths $P'$ and $P''$ is a directed path formed by important edges only, and
so the concatenation can be used to define the $i''$-parent of $w$, i.e., the vertex $w''$.
Hence, if $w'\in V_{\infty}$,
then the path $P'$ can be used to define the $i'$-parent $w'$ of $w$,
the path $P''$ can be used to define the $i$-parent of $w'$, and
their concatenation can be used to define the $i''$-parent of $w$, and
so the edge consistency follows as the edges are well-defined, which we have already established.
If $w'\in V_f$ (and so $w''\in V_f$),
then the path $P'$ can be used to define the $i'$-parent $w'$ of $w$, and
the definition of edges among the vertices of $V_f$ implies that 
the $i$-parent of $w'$ is the same vertex as
the $i''$-parent of $w$ defined by the concatenation of the paths $P'$ and $P''$.
Hence, the edge consistency holds in this case, too.

\textbf{Acyclicity.}
We next argue that the modeling $M$ is acyclic, i.e., it does not contain a finite directed cycle.
We first verify this for vertices contained in $V_f$.
Observe that the Hintikka chain of any vertex reachable (using the edges following their orientation)
from a vertex $(\Psi,m)\in V_f$ can never be $\Psi$.
The argument is similar as the one used when defining edges between vertices of $V_f$.
Indeed, if a vertex with Hintikka chain $\Psi$ can be reached from $(\Psi,m)$ by a directed path of length $k$,
then there exist $d$ and $n_0$ such that
every graph $G_n$ for $n\ge n_0$ has $\nu(\Psi)$ vertices whose $d$-Hintikka type is the $d$-formula of $\Psi$ and
there is a directed path from each such vertex to another vertex whose $d$-Hintikka type is the $d$-formula of $\Psi$;
this is impossible as any rooted $k$-tree is acyclic.
We conclude that if $M$ has a directed cycle, then it is comprised by vertices of $V_{\infty}$ only.

Observe that
if two vertices $(\Psi,n_0,h_1,n_1,\ldots,h_k,n_k)$ and $(\Psi',n'_0,h'_1,n'_1,\ldots,h'_k,n'_k)$ are joined by an edge,
then either $h_j=h'_j$ for all $j\in [k]$, or
there exists $j\in [k]$ such that $h'_j=h_j+\sqrt{2}\mod 1$ and $h'_{j'}=h_j$ for $j'=j+1,\ldots,k$, and
in addition, if $h_j=h'_j$ for all $j\in [k]$, then $n_j=n'_j$ for all $j\in [k]$.
Hence, if $M$ has a cycle formed by vertices of $V_{\infty}$,
then all vertices in this cycle have the same coordinates $h_j$, and so $n_j$ for all $j\in [k]$,
i.e., the last $2k$ coordinates of all vertices of the cycle are the same.
Consider such a cycle formed by vertices $(\Psi^1,n^1_0,h_1,n_1,\ldots,h_k,n_k),\ldots,(\Psi^\ell,n^\ell_0,h_1,n_1,\ldots,h_k,n_k)$.
However, this is only possible if every edge $e$ of the cycle satisfies the following:
$e$ is a finitary $i$-edge,
the head of $e$ has an $i'$-parent for $i'>i$ such that the edge from the head of $e$ to the $i'$-parent is finitary and important.
Let $i_0$ be the maximum of such $i'$ taken over all edges of the cycle.

Consider a vertex $v$ of the cycle such that the edge to its $i_0$-parent is finitary and important and
let $e$ be the edge of the cycle leading from $v$.
By the choice of $i_0$, the edge $e$ is an $i$-edge for $i<i_0$ (otherwise,
its head would have an $i'$-parent for $i'>i_0$ such that the edge to the $i'$-parent is finitary and important).
Since the edge from $v$ to its $i_0$-parent is important,
the $i_0$-parent of $v$ is also the $i_0$-parent of the head of $e$:
it is $i'$-parent for $i'\le i_0$ by Proposition~\ref{prop:parent} and
$i'=i_0$ as otherwise, there would be a detour for the edge from $v$ to its $i_0$-parent.
We conclude that there exists a vertex $w$ such that
$w$ is the $i_0$-parent of every vertex of the cycle and
the edge from each vertex of the cycle to $w$ is finitary (and important).

Let $\Psi$ be the Hintikka chain of any vertex of the cycle.
The construction of the modeling $M$ implies that
$w$ has an $i_0$-child with the Hintikka chain $\Psi$ and
there is a finite directed path to another $i_0$-child of $w$ with the Hintikka chain $\Psi$;
it follows that $w$ has infinitely many $i_0$-children with the Hintikka chain $\Psi$ and
so the edges from such $i_0$-children to $w$ cannot be finitary.
We conclude that $M$ has no directed cycle comprised by vertices of $V_{\infty}$ only.

\textbf{Measure and its properties.}
We now define a probability measure $\mu_M$ on $V_f\cup V_{\infty}$,
which is the probability measure on $M$ from the definition of a modeling.
We start by defining a topology on $V_f\cup V_{\infty}$:
a subset $X\subseteq V_f\cup V_{\infty}$ is open if
the set $X\cap V_{\infty}$
is open in the product topology of $\SSS_{\FOlocal}$ and $[0,1)^{2k+1}$.
In particular, every subset of the countable set $V_f$ is open.
Since $\SSS_{\FOlocal}$ restricted to $\nu^{-1}(\infty)$ (recall that the set $\nu^{-1}(\infty)$ is closed in $\SSS_{\FOlocal}$, which itself is a Polish space) and
$[0,1)^{2k+1}$ are Polish spaces, i.e., separable completely metrizable topological spaces,
their product is a Polish space.
Since the set $V_f$ is countable,
it follows that the Borel $\sigma$-algebra on $M$ is a standard Borel space.
Finally, the measure of a Borel set $X\subseteq V_f\cup V_{\infty}$ is the measure of the set $X\cap V_{\infty}$
given by the product measure determined by the measure $\mu$ (defined on the $\sigma$-algebra of Borel sets of $\SSS_{\FOlocal}$) and
the usual (Borel) measure on $[0,1)^{2k+1}$.

We next verify that $M$ is a modeling,
in particular, that all first order definable subsets of $M^k$ for every $k\in\NN$ are measurable, and
has the properties given in the statement of the theorem.
The edge consistency of $M$ implies that,
for every Hintikka chain $\Psi=(\psi_d)_{d\in\NN}$,
the vertices $v$ of $M$ such that the $d$-Hintikka type of $v$ is $\psi_d$ for all $d\in\NN$
are exactly the vertices whose first coordinate is $\Psi$,
i.e., vertices $(\Psi,i)\in V_f$ if $\nu(\Psi)\in\NN$ and
vertices $(\Psi,n_0,h_1,n_1,\ldots,h_k,n_k)\in V_{\infty}$ if $\nu(\Psi)=\infty$.
In particular, it holds that $\langle \psi_d,M\rangle=\mu(\psi_d)$ by the definition of the measure on $X\cap V_{\infty}$.
If $\psi$ is a local first order formula with a single free variable,
then the set of the Hintikka chains $\Psi$ of vertices $v$ of $M$ such that $M\models\psi(v)$
is a finite union of disjoint basic sets of Hintikka chains,
specifically, if $\psi$ has quantifier depth at most $d$ and uses only the first $d$ unary relations $U_i$,
then this is the set of Hintikka chains $\Psi$ such that the $d$-Hintikka type of $\Psi$ contains $\psi$.
This implies that the subset of the vertices $v$ of $M$ such that $M\models\psi(v)$
is open (in the topological space $V_f\cup V_{\infty}$) and so measurable, and
its measure is equal to the sum of the Stone measures of the $d$-Hintikka formulas that
are associated with the roots of the subtrees of $\TT_{\FOlocal}$ corresponding to these basic sets.
We also obtain that the discrete Stone measure and the Stone measure of the modeling $M$
are $\nu$ and $\mu$, respectively.

We now argue, using the way that $M$ was constructed, that
all first order definable subsets of $M^\ell$, $\ell\in\NN$, are measurable in the product measure and
the modeling $M$ satisfies the strong finitary mass transport principle.
This argument is analogous to the corresponding argument in the proof of~\cite[Lemma 40]{NesO16}
but we include it for completeness.
For $i\in [k]$,
let $g_i$ be the (partial) mapping from $V_f\cup V_{\infty}$ to $V_f\cup V_{\infty}$ that
maps $x\in V_f\cup V_{\infty}$ to the $i$-parent of $x$, and
let $g^0_i$ be the (partial) mapping from $V_f\cup V_{\infty}$ to $V_f\cup V_{\infty}$ that
maps $x\in V_f\cup V_{\infty}$ to the $i$-parent of $x$
whenever the $i$-parent of $x$ is contained in $V_f$ or
the $i$-parent of $x$ is contained in $V_{\infty}$ and
the edge joining $x$ to the $i$-parent is important.

Let $g'_i$ for $i\in [k]$ be the mapping from Hintikka chains to Hintikka chains such that
$g'_i(\Psi)$ is the Hintikka chain of the $i$-parent of a vertex described by $\Psi$.
Observe that the mapping $g'_i$
is continuous on the space of Hintikka chains with the topology defined in Subsection~\ref{subsec:Hintikka}.
Since the mappings $f(x):=x+\sqrt{2}\mod 1$, $f_i(x)=ix\mod 1$, $i\in\NN$, $g'_i$, $i\in [k]$, and $\zeta^d$, $d\in [2k]$,
are continuous (in the case of $f$ and $f_i$
when viewed as a function from $S_1=\mathbb{R}/\ZZ$ to $S_1=\mathbb{R}/\ZZ$ rather than $[0,1)$ to $[0,1)$),
the definition of the edges of $M$ yields that
each of the mappings $g^0_i$, $i\in [k]$, is measurable (in fact,
it would be continuous if we replaced the topology on $[0,1)^{2k+1}$ with that of $S_1^{2k+1}$).
Since $\mu$ is the Stone measure of a first order convergent sequence of rooted $2$-edge-colored $k$-trees,
the mapping $g'_i$ also satisfies the following:
if $Y$ is a measurable subset of Hintikka chains such that each element of $Y$ has exactly $d$ $i$-children,
then $\mu((g'_i)^{-1}(Y))=d\mu(Y)$.
As the mappings $f$, $f_i$ and $\zeta^d$ are measure preserving,
we deduce the following:
if $X$ is measurable subset of the domain of $g^0_i$ and
$Y$ is a measurable subset of the codomain of $g^0_i$ such that each vertex of $Y$ has exactly $d$ $i$-children contained in $X$,
then $\mu_M((g^0_i)^{-1}(Y)\cap X)=d\mu_M(Y)$.

We next show that each of the mappings $g_i$, $i\in [k]$, is measurable and
has a measure preserving-like property similar to that of $g^0_i$,
which is given at the end of the previous paragraph.
Fix $i\in [k]$, and let $V_i$ be the domain of $g^0_i$.
We next consider directed paths $P$ comprised of two or more edges;
similarly as when we defined edges of $M$,
we think of a directed path $P$ as a type of a path rather than an actual path in a modeling;
note that $P$ bears information which of its edges are $1$-edges, $2$-edges, etc.
Let $\preceq$ be the order on directed paths $P$ such that
the paths are first ordered by their lengths (shorter paths preceding longer paths) and
the paths with the same length are ordered lexicographically (as viewed as words
whose $i$-th letter is $m$ if the $i$-th edge of $P$ is an $m$-edge).
Let $V_{i,P}$ be the set containing those $x\in (V_f\cup V_{\infty})\setminus V_i$ such that
$M$ contains a directed path from $x$ to its $i$-parent formed by important edges only that
matches the types of edges in $P$
but $M$ contains no directed path from $x$ to its $i$-parent formed by important edges that
matches the types of edges in any path $P'\prec P$.
Since every subset of $M$ that can be defined by a local first order formula with a single free variable is measurable,
the set $V_{i,P}$ is measurable for every $i\in [k]$ and every $P$.
Observe that the mapping $g_i$ restricted to $V_{i,P}$ is a composition of $g^0_j$, $j\in [k]$,
in particular, the restriction of $g_i$ to $V_{i,P}$ is measurable and
has the property given at the end of the previous paragraph.
Since the sets $V_i$ and $V_{i,P}$ partition the domain of $g_i$,
the mapping $g_i$ is measurable and satisfies the following:
if $X$ and $Y$ are measurable subsets of $V_f\cup V_{\infty}$ such that
each vertex of $Y$ has exactly $d$ $i$-children contained in $X$,
then $\mu_M(g_i^{-1}(Y)\cap X)=d\mu_M(Y)$.
We refer to this property of the mapping $g_i$ as \emph{measure semipreserving}.

We now establish that every first order-definable subset of $M^\ell$ is Borel, i.e., measurable.
We have already argued that every subset of $M$ defined by a local first order formula $\psi$ is open.
Since the mapping $g_i$ is measurable,
the subset of $M^2$ given by the $i$-th child-parent relation is Borel for every $i\in [k]$, and
more generally any subset of $M^\ell$ that corresponds to an $\ell$-vertex path that
follows specified child-parent relations is Borel.
Gaifman's Locality Theorem with vertex neighborhoods replaced with Hintikka types (see the discussion before Theorem~\ref{thm:hanf})
yields that whether an $\ell$-tuple of vertices of $M$ satisfies a particular formula with $\ell$ free variables
depends on the Hintikka types of the $\ell$ vertices and the configuration formed by them,
i.e., the subgraph induced by their sufficiently large neighborhoods (in particular,
whether the distance of the vertices is small and if so what is their mutual position
in the neighborhood of each other).
If two vertices have small distance in a rooted $k$-tree,
then there exists a path between them that is initially formed by outgoing edges and then by ingoing edges only;
since the set of pairs of vertices joined by such a path form a projection of a Borel set,
the set of such pairs of vertices is Borel.
It follows that every first order-definable subset of $M^{\ell}$, $\ell\in\NN$,
is finite union and intersection of Borel sets;
the sets in the union and intersection determine Hintikka types of the vertices of the $\ell$-tuple and
the (non-)existence and the type of paths between them, and
so every first order-definable subset of $M^{\ell}$ is Borel.
Since every first order-definable subset of $M^\ell$ is Borel,
we have just finished establishing that $M$ is a modeling.

\textbf{Strong finitary mass transport principle.}
We now verify that the modeling $M$ satisfies the strong finitary mass transport principle,
i.e., we show that any two measurable subsets $A$ and $B$ of $M$ such that
each vertex of $A$ has at least $a$ neighbors in $B$ and
each vertex of $B$ has at most $b$ neighbors in $A$
satisfy that $a\mu_M(A)\le b\mu_M(B)$.
Fix such subsets $A$ and $B$ and the corresponding integers $a$ and $b$.
Since the measure of $V_f$ is zero, we can assume that $A\subseteq V_{\infty}$.
Similarly,
since the measure of vertices that are joined to any of their children by an infinitary edge is zero,
we can also assume that all edges directed from $B$ to $A$ are finitary.

For $i\in [k]$,
let $V'_i$ be the subset of $V_i$ formed by those $x\in V_i$ such that
the edge from $x$ to the $i$-parent of $x$ is infinitary.
In particular, $V'_i$ contains all $x\in V_i$ such that
the $i$-parent of $x$ is contained in $V_f$ (recall that $V_i$ is the domain of $g^0_i$).
Similarly,
we define $V'_{i,P}$ for a directed path $P$ with two or more edges
to be the subset of $V_i$ formed by those $x\in V_{i,P}$ such that
the directed path from $x$ to its $i$-parent described by $P$ contains an infinitary edge.
Note that $g^0_i$ has the following property for every integer $m\in\NN$:
if sets $X\subseteq V_i$ and $Y\subseteq V_f\cup V_{\infty}$ satisfy that
each vertex $y\in Y$ has at most $m$ $i$-children in $X$,
then the set $X\cap (g^0_i)^{-1}(Y)$ has measure zero.
Since the restriction of the mapping $g_i$ to $V_{i,P}$ is a composition of mappings $g^0_j$, $j\in [k]$,
the following holds for every subset $X\subseteq V'_{i,P}$ and $Y\subseteq V_f\cup V_{\infty}$:
if there exists an integer $m\in\NN$ such that
each vertex of $Y$ has at most $m$ $i$-children in $X$,
then the set $X\cap g_i^{-1}(Y)$ has zero measure.
Finally, since the sets $V_i$ and $V_{i,P}$ partition the domain of $g_i$ and
the edge from $x$ to its $i$-parent is infinitary iff $x$ is contained in $V'_i$ or $V'_{i,P}$ for some $P$,
we conclude that the following holds for every integer $m\in\NN$:
if $X$ is a subset of the domain of $g_i$ such that
the edge from each vertex of $X$ to its $i$-parent is infinitary and
$Y$ is a subset of $V_f\cup V_{\infty}$ such that
each vertex of $Y$ has at most $m$ $i$-children in $X$,
then the set $X\cap g_i^{-1}(Y)$ has zero measure.

For every $i\in [k]$,
let $A_i$ be the set of all $x\in A$ such that the edge from $x$ to its $i$-parent is infinitary.
Since each vertex of $B$ has at most $b$ $i$-children in $X$,
the set $A_i\cap g_i^{-1}(B)$ has zero measure.
Hence, we can assume that $A_i=\emptyset$ for every $i\in [k]$ (as the elements of each set $A_i$
can be removed from $A$ without altering its measure).
This yields that all edges of $M$ directed from the set $A$ to the set $B$ are finitary.
Since all edges directed from the set $A$ to the set $B$ are finitary and $A\subseteq V_{\infty}$,
no edge between $A$ and $B$ is incident with a vertex of $B\cap V_f$.
Hence, we can assume that $B\subseteq V_{\infty}$ in the rest.
Recall that
we have already argued that we can assume that all edges of $M$ directed from the set $B$ to the set $A$ are finitary.
Since both sets $A$ and $B$ are subsets of $V_{\infty}$,
all edges between $A$ and $B$ (in either direction) are finitary and
each of the mappings $g_i$, $i\in [k]$, is measure semipreserving,
it follows that $a\mu_M(A)\le b\mu_M(B)$.
Hence, the modeling $M$ has the strong finitary mass transport principle.

We remark that an analogous argument yields that the strong finitary mass transport principle is preserved
if all edges of $M$ that have one of the two colors are removed as
the restriction of each mapping $g_i$ to the vertices joined to the $i$-parent by an edge of the non-removed color
is also measure semipreserving.

\textbf{Residuality.}
It now remains only to verify the fourth property from the statement of the theorem.
Suppose that there exist a vertex $(\Psi,i)\in V_f$ and $r\in\NN$ such that
the $r$-neighborhood of $(\Psi,i)$ in $M\setminus\{c_i,i\in\NN\}$,
i.e., the set of vertices reachable from $(\Psi,i)$ by (not necessarily directed) paths of length at most $r$ that
do not contain any $c_i$, $i\in\NN$,
has a positive measure, say $\varepsilon>0$.
Let $\Psi=(\psi_j)_{j\in\NN}$.
Since $\nu(\Psi)$ is finite, there exists $\psi_j$ such that $\nu(\psi_j)=m$ and $m\in\NN$.
By the definition of a null-partitioned sequence of graphs,
there exist integers $n_0$ and $k_0$ such that
the $r$-neighborhood of each vertex in $G_n\setminus\{c_1,\ldots,c_{k_0}\}$, $n\ge n_0$,
contains at most $\varepsilon |G_n|/2m$ vertices.
In particular, this holds for the $m$ vertices satisfying $\psi_j$.
This implies that the Stone measure of vertices joined to one of the $m$ vertices satisfying $\psi_j$
by a (not necessarily directed) path of length at most $r$ that avoids $c_1,\ldots,c_{k_0}$
does not exceed $\varepsilon/2$ (note that this
property is first order expressible and therefore captured by the Stone measure).
Hence, the $r$-neighborhood of $(\Psi,i)$ in $M\setminus\{c_i,i\in\NN\}$ can have measure at most $\varepsilon/2$,
contrary to our assumption that its measure is $\varepsilon$.

Next, suppose that there exist a vertex $(\Psi,n_0,h_1,n_1,\ldots,h_k,n_k)\in V_\infty$ and an integer $r$ such that
the $r$-neighborhood of $(\Psi,n_0,h_1,n_1,\ldots,h_k,n_k)$ in $M\setminus\{c_i,i\in\NN\}$ has a positive measure.
First observe that
the $r$-neighborhood of $(\Psi,n_0,h_1,n_1,\ldots,h_k,n_k)$ in $M\setminus\{c_i,i\in\NN\}$
contains vertices of $V_\infty$ only (otherwise, $V_f$ would contain a vertex whose $2r$-neighborhood has a positive measure).
However, the definition of the modeling $M$ implies that the penultimate coordinates of all the vertices
in the $r$-neighborhood of $(\Psi,n_0,h_1,n_1,\ldots,h_k,n_k)$ are $h_k\pm\sqrt{2}i\mod 1$ for $i=-r,\ldots,+r$ (here,
we use that the $r$-neighborhood does not contain a vertex from $V_f$).
Since the set of all vertices of $V_{\infty}$ with the penultimate coordinate equal to $h_k\pm\sqrt{2}i\mod 1$, $i=-r,\ldots,+r$,
has measure zero (because of the set of possible values for this coordinate has measure zero in $[0,1)$),
the $r$-neighborhood of $(\Psi,n_0,h_1,n_1,\ldots,h_k,n_k)$ in $M\setminus\{c_i,i\in\NN\}$ also has measure zero.
\end{proof}

\section{Graphs with bounded tree-width}
\label{sec:main}

Theorem~\ref{thm:compose} almost readily yields our main result.

\setcounter{theorem}{\thestoremaintheorem}
\begin{theorem}
Let $k$ be a positive integer.
Every first-order convergent sequence of graphs with tree-width at most $k$
has a modeling limit satisfying the strong finitary mass transport principle.
\end{theorem}

\begin{proof}
Let $(H_n)_{n\in\NN}$ be a first-order convergent sequence of graphs with tree-width at most $k$.
As every graph with tree-width at most $k$ is a subgraph of a rooted $k$-tree (by Proposition~\ref{prop:tw}),
for every $n\in\NN$,
there exists a rooted $k$-tree $H'_n$ such that $H_n$ is a spanning subgraph of $H'_n$.
Color the edges of $H'_n$ with two colors representing whether the edge is also an edge of $H_n$ or not.
Hence, $(H'_n)_{n\in\NN}$ is a sequence of $2$-edge-colored rooted $k$-trees.
Let $(G_n)_{n\in\NN}$ be a subsequence of $(H'_n)_{n\in\NN}$ that is first order convergent.
By Lemma~\ref{lm:unary},
there exists a first order convergent null-partitioned sequence $(G'_n)_{n\in\NN}$
obtained from a subsequence of $(G_n)_{n\in\NN}$ by interpreting unary relational symbols $U_i$, $i\in\NN$.
Lemma~\ref{lm:modeling} and Theorem~\ref{thm:compose} yield that
the sequence $(G'_n)_{n\in\NN}$ has a modeling limit $M_G$ satisfying the properties given in Theorem~\ref{thm:compose}.
Let $M_H$ be the modeling obtained from $M_G$ by removing all relations except the binary relation encoding the edge-color that
represents edges of $H'_n$ that are also in $H_n$.
The modeling $M_H$ is a modeling limit of the sequence $(H_n)_{n\in\NN}$ and
it satisfies the strong finitary mass transport principle by the last point in the statement of Theorem~\ref{thm:compose}.
\end{proof}

\section*{Acknowledgement}

The authors would like to thank both anonymous reviewers for their detailed and insightful comments,
which have helped to significantly improve the presentation of the arguments given in the paper and
also to clarify several technical aspects of the presented proofs.

\bibliographystyle{bibstyle}
\bibliography{folim-tw}

\begin{thebibliography}{10}
\providecommand{\url}[1]{\texttt{#1}}
\providecommand{\urlprefix}{URL }
\providecommand{\eprint}[2][]{\url{#2}}

\bibitem{AldL07}
D.~Aldous and R.~Lyons: \emph{Processes on unimodular random networks},
  Electron. J. Probab. \textbf{12} (2007), 1454--1508.

\bibitem{BenS01}
I.~Benjamini and O.~Schramm: \emph{Recurrence of distributional limits of
  finite planar graphs}, Electron. J. Probab. \textbf{6} (2001), 1--13.

\bibitem{BolR11}
B.~Bollob{\'a}s and O.~Riordan: \emph{Sparse graphs: metrics and random
  models}, Random Structures Algorithms \textbf{39} (2011), 1--38.

\bibitem{BorCG17}
C.~Borgs, J.~Chayes and D.~Gamarnik: \emph{Convergent sequences of sparse
  graphs: A large deviations approach}, Random Structures Algorithms
  \textbf{51} (2017), 52--89.

\bibitem{BorCLSSV06}
C.~Borgs, J.~Chayes, L.~Lov{\'a}sz, V.~T. S{\'o}s, B.~Szegedy and
  K.~Vesztergombi: \emph{Graph limits and parameter testing}, Proc. 38th annual
  ACM Symposium on Theory of computing (STOC) (2006), 261--270.

\bibitem{BorCLSV08}
C.~Borgs, J.~T. Chayes, L.~Lov{\'a}sz, V.~T. S{\'o}s and K.~Vesztergombi:
  \emph{Convergent sequences of dense graphs {I}: {S}ubgraph frequencies,
  metric properties and testing}, Adv. Math. \textbf{219} (2008), 1801--1851.

\bibitem{BorCLSV12}
C.~Borgs, J.~T. Chayes, L.~Lov{\'a}sz, V.~T. S{\'o}s and K.~Vesztergombi:
  \emph{Convergent sequences of dense graphs {II}. {M}ultiway cuts and
  statistical physics}, Ann. Math. \textbf{176} (2012), 151--219.

\bibitem{ChaKNPSV20}
T.~Chan, D.~Kr\'{a}l', J.~A. Noel, Y.~Pehova, M.~Sharifzadeh and J.~Volec:
  \emph{Characterization of quasirandom permutations by a pattern sum}, Random
  Structures Algorithms \textbf{57} (2020), 920--939.

\bibitem{CorR20}
L.~N. Coregliano and A.~A. Razborov: \emph{Semantic limits of dense
  combinatorial objects}, Russian Math. Surveys \textbf{75} (2020), 627--723.

\bibitem{CorR23}
L.~N. Coregliano and A.~A. Razborov: \emph{Natural quasirandomness properties},
  Random Structures Algorithms \textbf{63} (2023), 624--688.

\bibitem{Die10}
R.~Diestel: Graph Theory, Graduate texts in mathematics 173, 2010.

\bibitem{EbbF05}
H.-D. Ebbinghaus and J.~Flum: Finite model theory, 2005.

\bibitem{Ele07}
G.~Elek: \emph{Note on limits of finite graphs}, Combinatorica \textbf{27}
  (2007), 503--507.

\bibitem{GajHKKKOOT17}
J.~Gajarsk{\'y}, P.~Hlin{\v{e}}n{\'y}, T.~Kaiser, D.~Kr{\'a}l', M.~Kupec,
  J.~Obdr{\v{z}}{\'a}lek, S.~Ordyniak and V.~T{\r u}ma: \emph{First order
  limits of sparse graphs: Plane trees and path-width}, Random Structures
  Algorithms \textbf{50} (2017), 612--635.

\bibitem{GleGKK15}
R.~Glebov, A.~Grzesik, T.~Klimo\v{s}ov\'a and D.~Kr{\'{a}}l': \emph{Finitely
  forcible graphons and permutons}, J. Combin. Theory Ser. {B} \textbf{110}
  (2015), 112--135.

\bibitem{HatLS14}
H.~Hatami, L.~Lov{\'a}sz and B.~Szegedy: \emph{Limits of locally--globally
  convergent graph sequences}, Geom. Funct. Anal. \textbf{24} (2014), 269--296.

\bibitem{HlaMPP15}
J.~Hladk{\'y}, A.~M{\'a}th{\'e}, V.~Patel and O.~Pikhurko: \emph{Poset limits
  can be totally ordered}, Trans. Amer. Math. Soc. \textbf{367} (2015),
  4319--4337.

\bibitem{HopKMRS13}
C.~Hoppen, Y.~Kohayakawa, C.~G. Moreira, B.~R{\'{a}}th and R.~M. Sampaio:
  \emph{Limits of permutation sequences}, J. Combin. Theory Ser. {B}
  \textbf{103} (2013), 93--113.

\bibitem{HopKMS11}
C.~Hoppen, Y.~Kohayakawa, C.~G. Moreira and R.~M. Sampaio: \emph{Testing
  permutation properties through subpermutations}, Theor. Comput. Sci.
  \textbf{412} (2011), 3555--3567.

\bibitem{HosNO17}
L.~Hosseini, J.~Ne{\v{s}}et{\v{r}}il and P.~Ossona~de Mendez: \emph{Limits of
  mappings}, European J. Combin. \textbf{66} (2017), 145--159.

\bibitem{Jan11}
S.~Janson: \emph{Poset limits and exchangeable random posets}, Combinatorica
  \textbf{31} (2011), 529--563.

\bibitem{KarKLM17}
F.~Kardo{\v{s}}, D.~{Kr\'al'}, A.~Liebenau and L.~Mach: \emph{First order
  convergence of matroids}, European J. Combin. \textbf{59} (2017), 150--168.

\bibitem{Kec95}
A.~S. Kechris: Classical Descriptive Set Theory, \emph{Graduate Texts in
  Mathematics}, volume 156, 1995.

\bibitem{KenKRW20}
R.~Kenyon, D.~Kr\'{a}l', C.~Radin and P.~Winkler: \emph{Permutations with fixed
  pattern densities}, Random Structures Algorithms \textbf{56} (2020),
  220--250.

\bibitem{KraP13}
D.~{Kr{\'a}l'} and O.~Pikhurko: \emph{Quasirandom permutations are
  characterized by 4-point densities}, Geom. Funct. Anal. \textbf{23} (2013),
  570--579.

\bibitem{Kur22}
M.~Kure{\v c}ka: \emph{Lower bound on the size of a quasirandom forcing set of
  permutations}, Combin. Probab. Comput. \textbf{31} (2022), 304--319.

\bibitem{Lov12}
L.~Lov\'asz: Large Networks and Graph Limits, Colloquium Publications 60, 2012.

\bibitem{LovS06}
L.~Lov\'asz and B.~Szegedy: \emph{Limits of dense graph sequences}, J. Combin.
  Theory Ser. B \textbf{96} (2006), 933--957.

\bibitem{LovS10}
L.~Lov{\'a}sz and B.~Szegedy: \emph{Testing properties of graphs and
  functions}, Israel J. Math. \textbf{178} (2010), 113--156.

\bibitem{NesO08a}
J.~Ne{\v{s}}et{\v{r}}il and P.~Ossona~de Mendez: \emph{Grad and classes with
  bounded expansion {I}. {D}ecompositions}, European J. Combin. \textbf{29}
  (2008), 760--776.

\bibitem{NesO08b}
J.~Ne{\v{s}}et{\v{r}}il and P.~Ossona~de Mendez: \emph{Grad and classes with
  bounded expansion {II}. {A}lgorithmic aspects}, European J. Combin.
  \textbf{29} (2008), 777--791.

\bibitem{NesO08c}
J.~Ne{\v{s}}et{\v{r}}il and P.~Ossona~de Mendez: \emph{Grad and classes with
  bounded expansion {III}. {R}estricted graph homomorphism dualities}, European
  J. Combin. \textbf{29} (2008), 1012--1024.

\bibitem{NesO11}
J.~Ne{\v{s}}et{\v{r}}il and P.~Ossona~de Mendez: \emph{On nowhere dense
  graphs}, European J. Combin. \textbf{32} (2011), 600--617.

\bibitem{NesO12}
J.~Ne{\v{s}}et{\v{r}}il and P.~Ossona~de Mendez: \emph{A model theory approach
  to structural limits}, Comment. Math. Univ. Carolinae \textbf{53} (2012),
  581--603.

\bibitem{NesO12book}
J.~Ne{\v{s}}et{\v{r}}il and P.~Ossona~de Mendez: Sparsity: graphs, structures,
  and algorithms, Algorithms and Combinatorics 28, 2012.

\bibitem{NesO16a}
J.~Ne{\v{s}}et{\v{r}}il and P.~Ossona~de Mendez: \emph{First-order limits, an
  analytical perspective}, European J. Combin. \textbf{52} (2016), 368--388.

\bibitem{NesO16}
J.~Ne{\v{s}}et{\v{r}}il and P.~Ossona~de Mendez: \emph{Modeling limits in
  hereditary classes: Reduction and application to trees}, Electron. J. Combin.
  \textbf{23} (2016), P2.52.

\bibitem{NesO16c}
J.~Ne{\v{s}}et{\v{r}}il and P.~Ossona~de Mendez: \emph{Structural sparsity},
  Uspekhi Mat. Nauk \textbf{71} (2016), 85--116.

\bibitem{NesO19}
J.~Ne{\v{s}}et{\v{r}}il and P.~Ossona~de Mendez: \emph{Existence of modeling
  limits for sequences of sparse structures}, J. Symbolic Logic \textbf{84}
  (2019), 452--472.

\bibitem{NesO20}
J.~Ne{\v{s}}et{\v{r}}il and P.~Ossona~de Mendez: \emph{A unified approach to
  structural limits and limits of graphs with bounded tree-depth}, Mem. Amer.
  Math. Soc. \textbf{263} (2020).

\bibitem{Tse19}
A.~Tserunyan: \emph{Introduction to descriptive set theory}, lecture notes at
  {\tt
  https://faculty.math.illinois.edu/\~{}anush/Teaching\_notes/dst\_lectures.pdf}
  (2019).

\end{thebibliography}
\end{document}